\documentclass[11pt]{amsart}

\usepackage[mathcal]{euscript}
\usepackage{amssymb, amsfonts, xypic, amsmath}
\usepackage{amscd}
\usepackage[all]{xy}
\usepackage{dutchcal}
\usepackage{amsthm}
\usepackage[margin=1.5in]{geometry}
\usepackage{xcolor}
\usepackage{tikz-cd}

\newtheorem{theorem}{Theorem}
\newtheorem{lemma}[theorem]{Lemma}
\newtheorem{corollary}[theorem]{Corollary}

\newtheorem{proposition}[theorem]{Proposition}
\newtheorem{question}[theorem]{Question}

\theoremstyle{definition}
\newtheorem{definition}[theorem]{Definition}

\newtheorem{remark}[theorem]{Remark}
\newtheorem*{theorem*}{Theorem}

\numberwithin{equation}{section} \numberwithin{figure}{section}

\numberwithin{equation}{section}
\DeclareRobustCommand{\rchi}{{\mathpalette\irchi\relax}}
\newcommand{\irchi}[2]{\raisebox{\depth}{$#1\chi$}}

\newcommand{\Q}{\mathbb{Q}}

\newcommand{\E}{\E_{\infty}}

\usepackage{graphicx}
\usepackage{epstopdf}

\author{Manuel Rivera, Felix Wierstra, Mahmoud Zeinalian}

\newcommand{\Addresses}{{
  \bigskip
  \footnotesize

\textsc{Manuel Rivera, Department of Mathematics, Purdue University, West Lafayette, IN 47907}\par\nopagebreak \textit{E-mail address} \texttt{manuelr@purdue.edu}
  
    \medskip
  \medskip 

  \textsc{Felix Wierstra, Department of Mathematics, Stockholm University, Kr\"aftriket 6, 106 91 Stockholm, Sweden}\par\nopagebreak
  \textit{E-mail address} \texttt{felix.wierstra@gmail.com}

  \medskip
  \medskip
  
  \textsc{Mahmoud Zeinalian, Department of Mathematics, City University of New York, Lehman College, 250 Bedford Park Blvd W, Bronx, NY 10468
   }\par\nopagebreak
  \textit{E-mail address} \texttt{mahmoud.zeinalian@lehman.cuny.edu}

}}

\begin{document}

\title[]{Rational homotopy equivalences and singular chains}
\maketitle

\begin{abstract}
Bousfield and Kan's  $\Q$-completion and fiberwise $\Q$-completion of spaces lead to two different approaches to the rational homotopy theory of non-simply connected spaces. In the first approach, a map is a weak equivalence if it induces an isomorphism on rational homology. In the second, a map of path-connected pointed spaces is a weak equivalence if it induces an isomorphism between fundamental groups and higher rationalized homotopy groups; we call these maps $\pi_1$-rational homotopy equivalences. In this paper, we compare these two notions and show that $\pi_1$-rational homotopy equivalences correspond to maps that induce $\Omega$-quasi-isomorphisms on the rational singular chains, i.e. maps that induce a quasi-isomorphism after applying the cobar functor to the dg coassociative coalgebra of rational singular chains. This implies that both notions of rational homotopy equivalence can be deduced from the rational singular chains by using different algebraic notions of weak equivalences: quasi-isomorphism and $\Omega$-quasi-isomorphisms. We further show that, in the second approach, there are no dg coalgebra models of the chains that are both strictly cocommutative and coassociative.

\end{abstract} 

\section{Introduction}

One of the questions that gave birth to rational homotopy theory is the \textit{commutative cochains problem} which, given a commutative ring $\mathbf{k}$, asks whether there exists a commutative differential graded (dg) associative $\mathbf{k}$-algebra functorially associated to any topological space that is weakly equivalent to the dg associative algebra of singular $\mathbf{k}$-cochains on the space with the cup product \cite{S77}, \cite{Q69}. In this paper we study a coalgebra version of this problem, which requires a careful consideration of what it means for two coalgebras to be weakly equivalent and for two possibly non-simply connected spaces to be rationally homotopy equivalent, as we now explain. 

For technical reasons we use a pointed version of the normalized singular chains and cochains (Definition \ref{def:Basedchains}). The \textit{pointed normalized singular chains} on a pointed space $(X,b)$, denoted by $C_*(X,b,\mathbf{k})$, form a connected coaugmented dg coassociative $\mathbf{k}$-coalgebra with the Alexander-Whitney diagonal approximation as coproduct. The linear dual of $C_*(X,b,\mathbf{k})$, denoted by $C^*(X,b,\mathbf{k})$, is the connected augmented dg associative $\mathbf{k}$-algebra of pointed normalized singular cochains with the cup product.
The Alexander-Whitney coproduct on $C_*(X,b;\mathbf{k})$ is not strictly cocommutative, but its cocommutativity holds up to an infinite coherent family of homotopies. This algebraic structure may be described using the language of operads; namely, the dg coassociative coalgebra structure on  $C_*(X,b;\mathbf{k})$ extends to an $E_{\infty}$-\textit{coalgebra} structure. This dualizes to an $E_{\infty}$-\textit{algebra} structure on $C^*(X,b;\mathbf{k})$ extending the cup product.

Denote by $\text{Top}_*$ the category of pointed path-connected topological spaces and by $\text{CDGA}_{\mathbf{k}}$ the category of augmented commutative dg associative $\mathbf{k}$-algebras. A pointed version of the commutative cochains problem is given by the following question.


\begin{question} Is there a functor $\mathcal{A}: \text{Top}_* \to \text{CDGA}_{\mathbf{k}}$  such that for any $(X,b) \in \text{Top}_*$, $\mathcal{A}(X,b)$ can be connected by a zig-zag of quasi-isomorphisms of augmented dg associative $\mathbf{k}$-algebras to $C^*(X,b; \mathbf{k})$?
\end{question}
Steenrod operations are obstructions for the existence of a functor $\mathcal{A}$ when $\mathbf{k}=\mathbb{Z}$ or a field of non-zero characteristic. Sullivan and Quillen showed via different approaches that when $\mathbf{k}=\mathbb{Q}$ such a functor $\mathcal{A}$ exists and, furthermore, the quasi-isomorphism type of the rational commutative dg algebra $\mathcal{A}(X,b)$ determines the rational homotopy type of $X$ when $X$ is a simply-connected space of finite type.

In \cite{BK71} and \cite{BK72}, Bousfield and Kan describe two possible rational completions for general (not necessarily nilpotent) spaces both leading to different extensions of the classical rational homotopy theory of Sullivan and Quillen. The first one, known as the $\Q$-\textit{completion} of a space, naturally associates to any space $X$ another space $\mathbb{Q}_{\infty}X$ together with a map $\rho: X \to \mathbb{Q}_{\infty}X$. When $X$ is a nilpotent space $\Q_{\infty}X$ has the Malcev completion of $\pi_1(X,b)$ as fundamental group and the rationalized higher homotopy groups of $X$ as higher homotopy groups.

We call a continuous map $f: X \to Y$ a \textit{$\Q_{\infty}$-homotopy equivalence} if $\Q_{\infty}f: \Q_{\infty}X \to \Q_{\infty}Y$ is a weak homotopy equivalence. This is the notion of weak equivalence in the extension of rational homotopy theory of \cite{BFMT18}. The $\Q$-completion construction satisfies the following properties. 
\begin{proposition} [\cite{BK71}] Let $f: (X,b) \to (Y,c)$ be a  continuous map of pointed path-connected spaces. Then, the following are equivalent:
\begin{enumerate}
\item The map $f: (X,b) \to (Y,c)$ is a $\Q_{\infty}$-homotopy equivalence, i.e. the induced map between $\Q$-completions $\Q_{\infty}f: \Q_{\infty}X \to \Q_{\infty}Y$ is a weak homotopy equivalence.
\item The induced map $C_*(f; \mathbb{Q}): C_*(X,b;\mathbb{Q}) \to C_*(Y,c; \mathbb{Q})$ is a quasi-isomorphism. 
\vspace{-8pt}
\item[] \hspace{-42pt} Furthermore, if $X$ and $Y$ are both nilpotent spaces then (1) and (2) are equivalent to:
\\
\item  The map $f: (X,b) \to (Y,c)$ induces an isomorphism between Malcev completions of the fundamental groups and an isomorphism of higher rationalized homotopy groups. 
\end{enumerate}

\end{proposition}
The second completion functor, known as the \emph{fiberwise $\Q$-completion}, associates to any space $X$ another space $\Q_\infty^*X$, having the same fundamental group as $X$ and higher homotopy groups isomorphic to the rationalized higher homotopy groups of $X$.  A continuous map $f: X \to Y$ is a \textit{$\pi_1$-rational homotopy equivalence} if $\Q_{\infty}^*f: \Q_{\infty}^*X \to \Q_{\infty}^*Y$ is a weak homotopy equivalence. This is the notion of weak equivalence in the extension of rational homotopy theory of \cite{GHT00}. 

In section 3, we prove that $\pi_1$-rational homotopy equivalences are detected by the dg coassociative coalgebra of pointed normalized singular chains with rational coefficients. Then, in section 4, we study a coalgebra version of Question 1 which fits with $\pi_1$-rational homotopy equivalences. More precisely, our first result is the following theorem.

\begin{theorem} \label{rationalchains} 
Let $f: (X,b) \to (Y,c)$ be a  continuous map of pointed path-connected spaces. Then, the following are equivalent:
\begin{enumerate}
 \item The map $f: (X,b) \to (Y,c)$ is a $\pi_1$-rational homotopy equivalence, i.e the induced map between fiberwise $\Q$-completions $\Q_{\infty}^*f: \Q_{\infty}^*X \to \Q_{\infty}^*Y$ is a weak homotopy equivalence.
 \item The induced maps $\pi_1(f): \pi_1(X,b) \to \pi_1(Y,c)$ and $\pi_n(f) \otimes \Q: \pi_n(X,b)\otimes \Q \to \pi_n(Y,c)\otimes \Q$ for $n \geq 2$ are isomor\-phisms.
 \item The induced map $C_*(f; \mathbb{Q}): C_*(X,b;\mathbb{Q}) \to C_*(Y,c; \mathbb{Q})$ is an $\Omega$-quasi-isomor\-phism, where an $\Omega$-quasi-isomorphism is a map that induces an isomorphism after applying the cobar construction (see section \ref{seccobar}).
 \item The induced map $\pi_1(f) :\pi_1(X,b) \rightarrow \pi_1(Y,c)$ is an isomorphism and for every $\Q$-representation $A$ of $\pi_1(Y,c)$ the induced map on the homology with local coefficients $H_*(f):H_*(X;f^*A)\rightarrow H_*(Y;A)$ is an isomorphism. 
\end{enumerate}
\end{theorem} 

The above theorem says that the notion of $\pi_1$-rational homotopy equivalence is in fact a rational notion, i.e. it can be described in terms of maps that preserve algebraic structures on rational vector spaces. The proof of the equivalence between (2) and (3) of Theorem \ref{rationalchains} uses a recent extension of a classical theorem of Adams, proven in \cite{RZ16} by the first and third author, which says that for any pointed \textit{path-connected} space $(X,b)$ there is a natural quasi-isomorphism of dg algebras $$\theta: \Omega C_*(X,b; \mathbb{Q}) \simeq C^{\square}_*(\Omega_bX; \mathbb{Q}),$$ where $C^{\square}_*(\Omega_bX; \mathbb{Q})$ is the dg algebra of rational cubical singular chains on the (Moore) based loop space of $X$ at $b$. The proof of Theorem \ref{rationalchains} also uses the fact that $\theta$ is a quasi-isomorphism of dg bialgebras for natural bialgebra structures on $\Omega C_*(X,b; \mathbb{Q})$ and $C^{\square}_*(\Omega_bX; \mathbb{Q})$. 

Quillen proved that associated to any simply connected space $X$ there is a rational cocommutative dg coassociative coalgebra which is quasi-isomorphic to the rational chains on $X$. Under the light of Theorem \ref{rationalchains} we may now ask a stronger question, namely, if for any path-connected $(X,b)$ the dg coalgebra $C_*(X,b;\mathbb{Q})$ may be strictified into a rational cocommutative dg coassociative coalgebra which still detects the fundamental group (or at least the fundamental group algebra) and the higher rational homotopy groups. More precisely, denoting by $\text{CDGC}_{\mathbf{k}}$ the category of coaugmented conilpotent dg cocommutative coassociative $\mathbf{k}$-coalgebras we ask the following question, which we call the \textit{cocommutative chains problem}:

\begin{question} \label{chainsproblem} Is there a functor $\mathcal{C}: \text{Top}_* \to \text{CDGC}_{\mathbf{k}}$ such that for any $(X,b) \in \text{Top}_*$, $\mathcal{C}(X,b)$ can be connected by a zig-zag of $\Omega$-quasi-isomorphisms of coaugmented conilpotent dg coassociative coalgebras to $C_*(X,b; \mathbf{k})$? 
\end{question}

Steenrod operations are also an obstruction for the existence of $\mathcal{C}$ when $\mathbf{k}= \mathbb{Z}$ or a field of non-zero characteristic. In section 4, we prove that there is no such functor $\mathcal{C}$ even when $\mathbf{k}$ is a field of characteristic zero. 

\subsection{Acknowledgments}

The first author acknowledges the support of the grant Fordecyt 265667 and the excellent working conditions of \textit{Centro de colaboraci\'on Samuel Gitler} in Mexico City. The second and third authors would like to thank the Max Planck Institute for Mathematics, where they first met and their collaboration started, for the hospitality and support during their stays.  

\section{Algebraic Preliminaries}

In this section we recall the algebraic constructions and results that will be used in the proofs of our main theorems in section 3 and 4. Let $\mathbf{k}$ be a commutative ring. We assume familiarity with the notions of differential graded (dg) $\mathbf{k}$-algebras and $\mathbf{k}$-coalgebras. The phrases ``dg algebra" and ``dg coalgebra" will mean ``unital augmented differential graded associative $\mathbf{k}$-algebra" and``counital coaugmented conilpotent differential graded coassociative $\mathbf{k}$-coalgebra", respectively. A graded $\mathbf{k}$-(co)algebra $V$ is \textit{connected} if $V_n=0$ for all $n <0$ and $V_0 \cong \mathbf{k}$. We also assume familiarity with bialgebras and Hopf algebras; in this paper we will mean by ``bialgebra" a unital augmented associative counital coaugmented coassociative bialgebra. A Hopf algebra is a bialgebra having the property of admitting an antipode. Antipodes are unique when they exist. We refer to \cite{LV12} for further background. 

\subsection{The cobar construction}\label{seccobar}

Let $\text{DGA}_{\mathbf{k}}$ and $\text{DGC}_{\mathbf{k}}$ denote the categories of dg algebras and dg coalgebras, respectively. Recall the definition of the \textit{cobar} functor
$$\Omega: \text{DGC}_{\mathbf{k}}  \to \text{DGA}_{\mathbf{k}}.$$
Given a dg coalgebra $C$ define a dg algebra
$$\Omega C := ( T(s^{-1}  \bar{C} ), D)$$
where $\bar C$ is the cokernel of the coaugmentation $\mathbf{k} \to C$, $s^{-1}$ is the shift functor which lowers degree by $1$, $T(s^{-1} \bar{C})= \mathbf{k} \oplus \bigoplus_{i=1}^{\infty} (s^{-1}\bar{C})^{\otimes i}$ the unital tensor algebra, and the differential $D$ is defined as follows. Let $\partial: C \to C$ and $\Delta: C \to C \otimes C$ be the differential and coproduct of $C$. Now extend the linear map $$- s^{-1} \circ \partial  \circ s^{+1} + (s^{-1} \otimes s^{-1}) \circ \Delta  \circ s^{+1}: s^{-1}\bar{C} \to T(s^{-1} \bar C)$$ as a derivation to obtain $D: T(s^{-1} \bar{C}) \to T(s^{-1} \bar{C})$. The coassociativity of $\Delta$, the compatibility of $\partial$ and $\Delta$, and the fact that $\partial^2 =0$ together imply that $D^2=0$. 

A \textit{quasi-isomorphism} of dg $\mathbf{k}$-modules is a chain map which induces an isomorphism in homology. A map of dg (co)algebras is said to be a quasi-isomorphism if the underlying map of dg $\mathbf{k}$-modules is. The following stronger notion of weak equivalence between dg coalgebras will play a fundamental role in this article. 
\begin{definition} A map of dg coalgebras $f: C\to C'$ is called an \textit{$\Omega$-quasi-isomorphism} if $\Omega f: \Omega C \to \Omega C'$ is a quasi-isomorphism of dg algebras. 
\end{definition} 

\begin{remark} Any $\Omega$-quasi-isomorphism is a quasi-isomorphism but not vice-versa, an example of this fact can be found in Proposition 2.4.3 of \cite{LV12}. Another example is given by considering the simplicial set $S$ which has exactly one vertex, one non-degenerate $1$-simplex, and all the higher simplices are degenerate. Since $|S|$ is homotopy equivalent to the circle $S^1$, the dg coalgebras of simplicial chains $C^{\Delta}_*(S;\mathbf{k})$ and singular chains $C_*(S^1;\mathbf{k})$ are quasi-isomorphic. However, $H_0( \Omega C^{\Delta}_*(S) )$ is isomorphic to the polynomial algebra $\mathbf{k}[x],$ while $H_0( \Omega C_*(S^1;\mathbf{k}) )$ is isomorphic to $\mathbf{k}[x,x^{-1}]$. 

A quasi-isomorphism between \text{simply connected dg coalgebras} (namely, non-negatively graded dg coalgebras $C$ such that $C_0 \cong \mathbf{k}$ and $C_1=0$) is an $\Omega$-quasi-isomorphism, as discussed in section 2.4 of \cite{LV12}. 
 \end{remark}

\subsection{Lie algebras and related constructions} 
For the rest of the section we work over a field $\mathbf{k}$ of characteristic zero. We review some classical constructions and results which will be used in Sections 3 and 4. We refer to \cite{Q69} or \cite{FHT01} for further details. 

Given a dg vector space $V$ denote by $SV$ the the \textit{symmetric algebra} generated by $V$. The dg commutative algebra $SV$ is defined as the quotient of the tensor algebra $TV$ by the ideal generated by elements of the form $x\otimes y - (-1)^{|x||y|} y\otimes x$ for $x,y \in V$. The unital dg associative algebra structure on $TV$ induces a unital commutative dg associative algebra structure on $SV$. Moreover, $SV$ is a commutative cocommutative dg bialgebra when equipped with coproduct $\Delta: SV \to SV \otimes SV$ given by extending $\Delta(x)=x \otimes 1 + 1 \otimes x$ as an algebra map and counit $S(V) \to \mathbf{k}$ induced by the projection $TV \to T^0V=\mathbf{k}$

Given a dg Lie algebra $L$ denote by $UL$ the \textit{universal enveloping algebra} of $L$. The dg associative algebra $UL$ is defined as the quotient of $TL$ by the ideal generated by elements of the form $x \otimes y - (-1)^{|x||y|}y \otimes x - [x,y]$ for $x,y \in L$. Moreover, $UL$ is a cocommutative dg bialgebra when equipped with the coproduct also determined by the formula $\Delta(x)=x \otimes 1 + 1\otimes x$ and counit induced by the projection $TV \to T^0V=\mathbf{k}$. 

The universal enveloping algebra construction defines a functor from dg Lie algebras to cocommutative dg bialgebras which commutes with homology.

\begin{theorem}[\cite{Q69} Appendix B, Proposition 2.1]\label{thrm:Universalenvelopecommuteswithhomology}
For any dg Lie algebra $L$ there is a natural isomorphism $$UH_*(L) \cong H_*(UL)$$ of graded cocommutative bialgebras. 
\end{theorem}

The Poincar\'e-Birkhoff-Witt (PBW) Theorem relates $S$ and $U$.

\begin{theorem}[Poincar\'e-Birkhoff-Witt, \cite{Q69} Appendix B, Theorem 2.3] \label{thrm:PBW} For any dg Lie algebra $L$ there is an isomorphism of cocommutative dg coalgebras $$e: SL \xrightarrow{\cong} UL.$$
\end{theorem}

Since the symmetric algebra functor preserves quasi-isomorphisms we immediately obtain the following corollary.

\begin{corollary}\label{cor:universalenvelopingalgebrapreservesqi}
 If $f:L\rightarrow L'$ is a quasi-isomorphism of dg Lie algebras, then $U f:U L \rightarrow U L'$ is a quasi-isomorphism as well. 
\end{corollary}
In \cite{Q69}, Quillen introduced a functor $$\mathcal{L}: \text{CDGC}_{\mathbf{k}} \to \text{DGL}_{\mathbf{k}},$$ where $\text{DGL}_{\mathbf{k}}$ denotes the category of dg Lie algebras, defined as follows. Given $C \in \text{CDGC}_{\mathbf{k}}$ with differential $\partial: C \to C$, the dg Lie algebra $\mathcal{L}C$ has as underlying graded Lie algebra the free graded Lie algebra $(L(s^{-1} \bar{C}), [ \_ , \_])$ generated by $s^{-1} \bar{C}$, the desuspension of the coaugmentation ideal of $C$. The differential $d_{\mathcal{L}C}$ is determined by
$$ d_{\mathcal{L}C} (\tau x) = - \tau \partial x - \frac{1}{2} \sum_{(x)} (-1)^{|x'|}[\tau x', \tau x''],$$
where $\tau: C \to L(s^{-1} \bar{C})$ is the canonical map given by $\tau x= s^{-1} \bar x$ and $\Delta(x)= \sum_{(x)} x' \otimes x''$ denotes the coproduct of $C$. 

The functors $\Omega$ and $\mathcal{L}$ are related via $U$ as follows.

\begin{lemma}[\cite{Q69} p.290]\label{lem:cobaranduniversalenvelopingalgebra}
 For any cocommutative dg coalgebra $C$ there is a natural isomorphism of augmented dg associative algebras $$\Omega C \cong U \mathcal{L} C.$$
 \end{lemma}

We finish this section recalling the Milnor-Moore Theorem in topology.

\begin{theorem}[Milnor-Moore, \cite{FHT01} Theorem 21.5] \label{thrm:milnor-moore}
Let $(X,b)$ be a simply connected pointed path-connected space. There is an isomorphism of graded Hopf algebras
$$U(\pi_*(\Omega_bX) \otimes \mathbf{k}) \cong H_*(\Omega_bX;\mathbf{k}),$$
where the graded Lie algebra structure on $\pi_*(\Omega_bX) \otimes \mathbf{k}$ is given by the Whitehead bracket. 
\end{theorem}

\section{coalgebras and non-simply-connected rational homotopy theory }\label{sec:nonsimplyconnectedrationalhomotopytheory}

 In this section we define two notions of weak equivalences between spaces for two different approaches to extend rational homotopy theory to non-simply connected spaces. These are based on two rational completions for spaces described in \cite{BK71}. We explain the sense in which the rational dg coalgebra of pointed normalized singular chains detects each of these weak equivalences. All topological spaces will be assumed to be semi-locally simply connected and locally path-connected so that universal covers exist. Denote by $\text{Top}_*$ the category of pointed path-connected spaces. We first define the version of singular chains on pointed spaces that will be used. 

\begin{definition}\label{def:Basedchains}
For any $(X,b) \in \text{Top}_*$ and any commutative ring $\mathbf{k}$, the coaugmented connected dg $\mathbf{k}$-coalgebra of \textit{pointed normalized singular chains} $(C_*(X,b;\mathbf{k}), \partial, \Delta)$ is defined as follows. The underlying graded $\mathbf{k}$-module is obtained by considering the graded submodule of the ordinary singular chains generated by those continuous maps $\sigma: \Delta^n \to X$ that send the vertices of the $n$-simplex $\Delta^n$ to $b\in X$, and then modding out by degenerate simplices. The usual boundary operator for singular chains induces a differential $\partial: C_*(X,b;\mathbf{k})\to C_{*-1}(X,b;\mathbf{k})$ and the Alexander-Whitney diagonal approximation map induces a compatible coassociative coproduct $$\Delta: C_*(X,b;\mathbf{k})\to C_*(X,b;\mathbf{k})\otimes C_*(X,b;\mathbf{k}).$$ Note that $C_0(X,b;\mathbf{k}) \cong \mathbf{k}$, so the counit and coaugmentation are canonically defined. This construction defines a functor $$C_*: \text{Top}_* \to \text{DGC}^0_{\mathbf{k}},$$ where $\text{DGC}^0_{\mathbf{k}}$ denotes the full subcategory of $\text{DGC}_{\mathbf{k}}$ whose objects are connected dg coalgebras.
\end{definition}

In  \cite{BK71}, Bousfield and Kan define the  $\mathbb{Q}$-\textit{completion} of a space $X$ as $$\Q_{\infty}X =|\overline{\mathbf{W}} \Q_{\infty} (\mathbf{G} \text{Sing}(X,b) ) |,$$
where $\text{Sing}(X,b)$ is the subsimiplicial set of \text{Sing}(X) consisting of those singular simplices $\Delta^n \to X$ that send the vertices of $\Delta^n$ to $b$, $\mathbf{G}: \text{sSet} \to \text{sGrp}$ is the Kan loop group functor from simplicial sets to simplicial groups, $\Q_{\infty} G$ denotes the dimensionwise (algebraic) $\Q$-completion of any simplicial group $G$, $\overline{\mathbf{W}}: \text{sGrp} \to \text{sSet}$ is the classifying space functor, and $| \cdot |$ denotes geometric realization. Note we have used a slightly different notation from that in \cite{BK71}. 

\begin{definition} 
A continuous map $f:X \rightarrow Y$ between path-connected spaces is a \textit{$\Q_{\infty}$-homotopy equivalence} if $\Q_\infty f : \Q_\infty X \rightarrow \Q_\infty Y$, the induced map on the $\Q$-completions, is a weak homotopy equivalence. 
\end{definition}

The rational singular chains are able to detect $\Q_{\infty}$-homotopy equivalences in the following sense.

\begin{proposition} \label{BKeq} \cite{BK71} Let $f: (X,b) \to (Y,c)$ be a  continuous map of pointed path-connected spaces. Then, the following are equivalent:
\begin{enumerate}
\item The map $f: (X,b) \to (Y,c)$ is a $\Q_{\infty}$-homotopy equivalence, i.e. the induced map between $\Q$-completions $\Q_{\infty}f: \Q_{\infty}X \to \Q_{\infty}Y$ is a weak homotopy equivalence.
\item The induced map $C_*(f; \mathbb{Q}): C_*(X,b;\mathbb{Q}) \to C_*(Y,c; \mathbb{Q})$ is a quasi-isomorphism. 
\vspace{-8pt}
\item[] \hspace{-42pt} Furthermore, if $X$ and $Y$ are both nilpotent spaces then (1) and (2) are equivalent to
\vspace{4pt}
\item  $f: (X,b) \to (Y,c)$ induces an isomorphism between Malcev completions of the fundamental groups and an isomorphism of higher rationalized homotopy groups. 
\end{enumerate}
\end{proposition}

\begin{proof} The equivalence between (1) and (2) is exactly 5.2 in \cite{BK71}. If $X$ is a nilpotent space then, by combining the results of \cite{BK71}, we see that  $\rho: X \to \Q_{\infty}X$ induces an isomorphism on the Malcev completions of the fundamental groups and an isomorphism $\pi_n(\rho) \otimes \Q: \pi_n(X) \otimes \Q \cong \pi_n(\Q_{\infty}X)$ for all $n \geq 2$. This is because nilpotent spaces are $\Q$-good and the $\Q$-completion of a nilpotent group is given by the Malcev completion. It therefore follows that if $X$ and $Y$ are both nilpotent then (3) is equivalent to (1). 
\end{proof}

In \cite{BK71}, Bousfield and Kan also define a second possibility of completion for non-simply-connected spaces called \emph{fiberwise $\Q$-completion}. This is done by fiberwise $\Q$-completing the fibration $\tilde{X} \to X  \to B\pi_1(X,b)$, where $\tilde{X}$ denotes the universal cover of $X$ and $B\pi_1(X,b)$ is the classifying space of $\pi_1(X,b)$. The fiberwise $\Q$-completion uses the algebraic \textit{relative $\Q$-completion} for a short exact sequence of groups. This results in a fibration $$\mathbb{Q}_{\infty}\tilde{X} \to \mathbb{Q}^*_{\infty}X \to B\pi_1(X,b),$$ naturally associated to any $X$, whose fiber is the ordinary $\Q$-completion of $\tilde{X}$. Thus, to any space $X$ we may functorially associate a new space $\mathbb{Q}^*_{\infty}X$ which has the same fundamental group as $X$ and whose higher homotopy groups are the rationalized homotopy groups of $X$. 

\begin{definition}\label{def:ratequivalence}
A continuous map $f:X \rightarrow Y$ between path-connected spaces is a \textit{$\pi_1$-rational homotopy equivalence} if $\Q_\infty ^*f : \Q_\infty^* X \rightarrow \Q_\infty^* Y$, the induced map on the fiberwise $\Q$-completions, is a weak homotopy equivalence. 
\end{definition}
Our next result gives three alternative characterizations of $\pi_1$-rational homotopy equivalences in terms of the coalgebras of pointed normalized singular chains, homotopy groups, and homology with local coefficients.

\begin{theorem} \label{coalgebrasandrationaleqs}
Let $f: (X,b) \to (Y,c)$ be a  continuous map of pointed path-connected spaces. Then, the following are equivalent
\begin{enumerate}
 \item The map $f: (X,b) \to (Y,c)$ is a $\pi_1$-rational homotopy equivalence, i.e the induced map between fiberwise $\Q$-completions $\Q_{\infty}^*f: \Q_{\infty}^*X \to \Q_{\infty}^*Y$ is a weak homotopy equivalence.
 \item The induced maps $\pi_1(f): \pi_1(X,b) \to \pi_1(Y,c)$ and $\pi_n(f) \otimes \Q: \pi_n(X,b)\otimes \Q \to \pi_n(Y,c)\otimes \Q$ for $n \geq 2$ are isomorphisms.
  \item The induced map $C_*(f; \mathbb{Q}): C_*(X,b;\mathbb{Q}) \to C_*(Y,c; \mathbb{Q})$ is an $\Omega$-quasi-iso\-mor\-phism.
 \item The induced map $\pi_1(f) :\pi_1(X,b) \rightarrow \pi_1(Y,c)$ is an isomorphism and for every $\Q$-representation $A$ of $\pi_1(Y,c)$ the induced map on the homology with local coefficients $H_*(f):H_*(X;f^*A)\rightarrow H_*(Y;A)$ is an isomorphism. 
\end{enumerate}
\end{theorem} 

In the above theorem, by a \textit{$\Q$-representation} of a group $G$ we mean a $\mathbb{Q}$-vector space $A$ together with a is a left $\mathbb{Q}[G]$-module structure. If $f: G'\to G$ is a group homomorphism and $A$ is a $\Q$-representation of $G$, then $f^*A$ denotes the $\Q$-representation of $G'$ given by the pullback action of $\mathbb{Q}[G']$ on $A$ via $f$. 

\begin{remark}
A direct observation that follows from Proposition \ref{BKeq} and Theorem \ref{coalgebrasandrationaleqs} is that both notions of rational homotopy equivalence can be deduced from the pointed normalized singular chains with rational coefficients. This might come as a surprise, since this means that the singular chains with rational chains are capable of capturing the highly non-rational fundamental group in the case of $\pi_1$-rational homotopy equivalences and the Malcev completion of the fundamental group in the case of $\Q_{\infty}$-homotopy equivalences. This can be interpreted as saying that both approaches to non-simply-connected rational homotopy theory are in fact rational, i.e. can be deduced from algebraic structure on a rational chain complex. We would further like to point out that both notions of rational homotopy equivalence coincide for simply connected spaces.
\end{remark}

Before we prove Theorem \ref{coalgebrasandrationaleqs} we first recall two results proven in \cite{RZ16} and \cite{RZ18} regarding an extension of the classical Adams' cobar theorem \cite{A56}. For completeness, we will sketch a proof of the following theorem and for a detailed proof we refer the reader to \cite{RZ18} or \cite{Ri19}. Let $\Omega_bX$ denote the topological monoid of Moore loops in $X$ based at $b \in X$. 

\begin{theorem} \label{thrm:adamsbialgebra}
Let $(X,b)$ be a pointed path-connected space. Then,
\begin{enumerate}
\item There is a natural quasi-isomorphism of dg algebras $$\theta: \Omega C_*(X,b;\mathbb{Q}) \to C^{\square}_*(\Omega_bX;\mathbb{Q}),$$ where $C^{\square}_*(\Omega_bX;\mathbb{Q})$ denotes the normalized singular cubical chains on $\Omega_bX$ with rational coefficients. 
\item There is a natural coassociative coproduct  $$\nabla: \Omega C_*(X,b;\mathbb{Q}) \to \Omega C_*(X,b;\mathbb{Q}) \otimes \Omega C_*(X,b;\mathbb{Q})$$ making $\Omega C_*(X,b;\mathbb{Q})$ a dg bialgebra such that $\theta$ becomes a quasi-isomorphism of dg bialgebras, when $C^{\square}_*(\Omega_bX;\mathbb{Q})$ is equipped with the natural diagonal approximation coproduct for cubical chains. 
\end{enumerate}
\end{theorem}
\begin{proof}[Sketch of proof.] ${}$
\begin{enumerate}
\item The fact that there exists a quasi-isomorphism of dg algebras $\theta: \Omega C_*(X,b;\mathbb{Q})$ $\to C^{\square}_*(\Omega_bX;\mathbb{Q})$ is an extension of a classical theorem of Adams proven in \cite{RZ16} by observing that $\Omega C_*(X,b;\mathbb{Q})$ is naturally isomorphic as a dg algebra to the chains on a monoidal cubical set with connections (denoted by $\mathfrak{C}_{\square_c}(\text{Sing}(X,b))$ in \cite{RZ16} and \cite{RZ18}), whose geometric realization is naturally homotopy equivalent to the based loop space. This natural homotopy equivalence induces a quasi-isomorphism from the dg algebra of chains on $\mathfrak{C}_{\square_c}(\text{Sing}(X,b))$ to $C^{\square}_*(\Omega_bX;\mathbb{Q})$.
\item The chain complex of normalized cubical chains on any cubical set (with or without connections) has a natural coassociative coproduct approximating the diagonal map. The quasi-isomorphism from the cubical chains on $\mathfrak{C}_{\square_c}(\text{Sing}(X,b))$ to $C^{\square}_*(\Omega_bX;\mathbb{Q})$ preserves coproducts. To construct $\nabla$, we transfer the coproduct of the cubical chains on  $\mathfrak{C}_{\square_c}(\text{Sing}(X,b))$ to $\Omega C_*(X,b;\mathbb{Q})$ via the isomorphism between them. 
\end{enumerate}
\end{proof}

\begin{remark} \label{surjection} The construction of the coproduct on $\Omega C_*(X,b;\mathbb{Q})$ builds up on an idea originally described in \cite{B98}. It is related to the $E_{\infty}$-coalgebra structure on $C_*(X,b;\mathbb{Q})$ as follows. It is shown in \cite{BF04} and \cite{MS02} that $C_*(X,b;\mathbb{Q})$ has a natural structure of a coalgebra over the \textit{surjection operad}, usually denoted by $\rchi$, extending the dg coassociative coalgebra structure given by the Alexander-Whitney diagonal approximation. The surjection operad $\rchi$ is a particular model for the $E_{\infty}$-operad. It is explained in \cite{K03} that for any connected $\rchi$-coalgebra $\mathbf{C}$ with underlying dg coassociative coalgebra $C$, the structure maps of the $\rchi$-coalgebra structure corresponding to the $E_2$ portion of the operad induce a dg bialgebra structure on $\Omega C$. In the case of $C_*(X,b;\mathbb{Q})$, the coproduct of the dg bialgebra structure on $\Omega C_*(X,b;\mathbb{Q})$ coincides with the coproduct $\nabla$ outlined in the proof of Theorem \ref{thrm:adamsbialgebra}. More details may be found in Theorems 2 and 3 of \cite{RZ18}.
\end{remark}

\begin{remark} Two other extensions of Adams' cobar theorem to the non-simply connected case may be found in \cite{FHT92} and \cite{HT10}. These approaches add formal inverses for the $1$-simplices in different ways. In \cite{FHT92}, an extension of Adams' construction is described for non-simply connected CW-complexes by adding homotopy inverses for all $1$-simplices together with formal homotopies at the level of the cellular chains. In \cite{HT10}, a different extension is described for any simplicial set by adding formal (strict) inverses for all $1$-simplices after applying the cobar construction. In the approach of the first and third authors in \cite{RZ16} no inverses must be added since the construction is preformed at the level of the Kan complex of singular chains, which already contains inverses up to homotopy. 
\end{remark}

Since there is an isomorphism of bialgebras $H_0(\Omega_b X;\mathbb{Q}) \cong \mathbb{Q} [\pi_1(X,b)]$ and any group algebra is a Hopf algebra (a bialgebra which has an antipode map) whose group-like elements are the underlying group, we immediately obtain the following corollary. 

\begin{corollary}\label{fundgroup} The map $\theta: \Omega C_*(X,b;\mathbb{Q}) \to C^{\square}_*(\Omega_bX;\mathbb{Q})$ induces an isomorphism of bialgebras $H_0(\Omega C_*(X,b;\mathbb{Q})) \cong \mathbb{Q}[\pi_1(X,b)]$. In particular, there exists an antipode on the bialgebra $H_0(\Omega C_*(X,b;\mathbb{Q}))$, which makes it a Hopf algebra. The isomorphism class of the fundamental group of $X$ may be recovered functorially as the group of group-like elements of $H_0(\Omega C_*(X,b;\mathbb{Q}))$.
\end{corollary}

We now use the above results to prove Theorem \ref{coalgebrasandrationaleqs}. An integral version of the equivalence between (2), (3), and (4)  of Theorem \ref{coalgebrasandrationaleqs} was proven in \cite{RWZ18}. 

\begin{proof}[Proof of Theorem \ref{coalgebrasandrationaleqs}.] The equivalence between (1) and (2) follows directly from the fact that $\Q_{\infty}^*X$ is the total space of a fibration $\Q_{\infty}\tilde{X} \to \Q_{\infty}^*X \to B\pi_1(X,b)$, since then we can use the long exact sequence for a fibration together the natural isomorphisms $\pi_n(\Q_{\infty}\tilde{X}) \cong \pi_n(X) \otimes \Q$. 

 The proof that (3) implies (4) is similar to the proof of Theorem 12 in \cite{RWZ18} as we now explain. If $C_*(f; \mathbb{Q}): C_*(X,b;\mathbb{Q}) \to C_*(Y,c;\mathbb{Q})$ is an $\Omega$-quasi-isomorphism, then $f$ induces an isomorphism of Hopf algebras $$H_0(\Omega(f)): H_0(\Omega C_*(X,b; \mathbb{Q})) \xrightarrow{\cong} H_0 (\Omega C_*(Y,c;\mathbb{Q}))$$ and, by Corollary \ref{fundgroup}, after applying the group-like elements functor to the isomorphism $H_0(\Omega(f))$ we obtain that $\pi_1(f): \pi_1(X,b) \to \pi_1(Y,c)$ is an isomorphism. It follows from Proposition 10 of \cite{RWZ18} that if $C_*(f): C_*(X,b;\mathbb{Q}) \to  C_*(Y,c;\mathbb{Q})$ is an $\Omega$-quasi-isomorphism, then the induced map on homology with coefficients in any rational local system is an isomorphism. 
 
To prove (4) implies (2) we consider the pointed universal covers of the pointed spaces $(X,b)$ and $(Y,c)$, which we denote by $(\tilde{X}, [b]) $ and $(\tilde{Y}, [c])$, respectively. By a standard lifting argument we get an induced map $\tilde{f}: (\tilde{X}, [b]) \rightarrow (\tilde{Y}, [c])$, which is unique up to homotopy. As explained in Section 5.2 of \cite{DK01}, the rational homology of the universal cover of $X$ can be computed by the homology of $X$ with local coefficients in the fundamental group algebra, i.e.  $H_*(\tilde{X};\Q)\cong H_*(X;\Q[\pi_1(X,b)])$, where $\Q[\pi_1(X,b)]$ is the representation of $\pi_1(X, b)$ through the left multiplication of $\pi_1(X, b)$ on itself. By assumption, this implies that $\tilde{f}$ induces an isomorphism on the rational homology of the universal covers. Consequently, since the universal covers are simply connected, Whitehead's Theorem yields that $\pi_n(\tilde{f}) \otimes \mathbb{Q}:\pi_n(\tilde{X}, [b])\otimes\Q \rightarrow \pi_n(\tilde{Y},[c]) \otimes \Q$ are isomorphisms for all $n \geq 2$.  It then follows from the long exact sequence in homotopy that $\pi_n(f) \otimes \mathbb{Q} :\pi_n(X,b)\otimes\Q \rightarrow \pi_n(Y,c) \otimes \Q$ are isomorphisms for all $n \geq 2$ as well.

We now show (2) implies (3). Suppose the maps $\pi_1(f):\pi_1(X,b) \rightarrow \pi_1(Y,c)$ and $\pi_n(f) \otimes \mathbb{Q} :\pi_n(X,b) \otimes \Q \rightarrow \pi_n(Y,c)\otimes \Q$ for $n \geq 2$ are isomorphisms. By Theorem \ref{thrm:adamsbialgebra} this is equivalent to show that $f$ induces an isomorphism on homology for the corresponding loop spaces. Again we will look at the universal covers to deduce the results. In particular, the lift $\tilde{f}$ induces an isomorphism $$\pi_n(\tilde{f}) \otimes \mathbb{Q}: \pi_n(\tilde{X}) \otimes \mathbb{Q} \xrightarrow{\cong} \pi_n(\tilde{Y}) \otimes \mathbb{Q}$$ for each $n \geq 1$. Therefore $$\pi_{n-1}(\Omega(\tilde{f})) \otimes \mathbb{Q}: \pi_{n-1}(\Omega_{[b]} \tilde{X}) \otimes \mathbb{Q} \xrightarrow{\cong} \pi_{n-1}(\Omega_{[c]} \tilde{Y}) \otimes \mathbb{Q}$$ are isomorphisms for each $n \geq 1$. This means that the map $\pi_n(\Omega(\tilde{f})) \otimes \mathbb{Q}$ is an isomorphism of Lie algebras and applying the universal enveloping algebra functor to each of these maps, by the Milnor-Moore Theorem (Theorem \ref{thrm:milnor-moore}), we get an isomorphism $$H_*(\Omega \tilde{f}): H_*(\Omega_{[b]} \tilde{X}; \mathbb{Q}) \xrightarrow{\cong} H_*(\Omega_{[c]} \tilde{Y}; \mathbb{Q}).$$ Recall there is a homotopy equivalence $ \Omega_b X \simeq \Omega_{[b]} \tilde{X}\times \pi_1(X,b)^{disc}$, where $\pi_1(X,b)^{disc}$ denotes $\pi_1(X,b)$ considered as a discrete space,  inducing a natural isomorphism of graded vector spaces $$H_*(\Omega_bX;\mathbb{Q}) \cong H_*(\Omega_{[b]} \tilde{X};\mathbb{Q}) \otimes H_*( \pi_1(X,b)^{disc};\mathbb{Q}).$$ Similarly, $$H_*(\Omega_cY;\mathbb{Q}) \cong H_*(\Omega_{[c]} \tilde{Y};\mathbb{Q}) \otimes H_*( \pi_1(Y,c)^{disc};\mathbb{Q}).$$ The latter two isomorphisms together with the fact that $H_*(\Omega \tilde{f})$ and $\pi_1(f)$ are isomorphisms imply that $H_*( \Omega (f) ): H_*( \Omega_bX; \mathbb{Q} ) \to H_*(\Omega_cY; \mathbb{Q})$ is an isomorphism as well. Finally, by Theorem \ref{thrm:adamsbialgebra}, it follows that $\Omega C_*(f): \Omega C_*(X,b;\mathbb{Q}) \to \Omega C_*(Y,c;\mathbb{Q})$ is a quasi-isomorphism, as desired. 
\end{proof}

\section{The cocommutative chains problem}

 The goal of this section is to give a negative answer to Question \ref{chainsproblem}. If we restrict our attention to simply connected spaces and coalgebras with quasi-isomorphisms as weak equivalences then Quillen's rational homotopy theory provides a positive answer to the cocommutative chains problem: naturally associated to any simply connected space there is a dg Lie algebra whose Chevalley-Eilenberg complex is a cocommutative dg coassociative coalgebra quasi-isomorphic to the rational chains on the space \cite{Q69}. However, if we consider the problem for all path-connected spaces and for coalgebras with $\Omega$-quasi-isomorphisms as weak equivalences, as asked in Question \ref{chainsproblem}, the answer turns out to be negative. 
 
Suppose there is a cocommutative dg coalgebra $\mathcal{C}(X,b)$ such that $\Omega \mathcal{C}(X,b)$ is quasi-isomorphic as a dg algebra to $\Omega C_*(X,b; \mathbf{k})$. By Corollary \ref{fundgroup},  $H_0(\Omega C_*(X,b; \mathbf{k})) \cong \mathbf{k}[\pi_1(X,b)]$, so for a hypothetical strictly cocommutative model $\mathcal{C}(X,b)$ we would also have an isomorphism of algebras $H_0(\Omega \mathcal{C}(X,b)) \cong \mathbf{k}[\pi_1(X,b)]$. In the proof of the upcoming theorem we show that, if $C$ is a cocommutative dg coalgebra, $H_0(\Omega C)$ is always isomorphic as a vector space to a polynomial algebra. Therefore $H_0(\Omega C)$ is always infinite dimensional which means it cannot be isomorphic to the group algebra of a finite group. 

Recall $\text{CDGC}_{\mathbf{k}}$ denotes the category of cocommutative dg  $\mathbf{k}$-coalgebras and $\text{Top}_*$ the category of pointed path-connected spaces. 

\begin{theorem}\label{thrm21} Let $\mathbf{k}$ be a field of characteristic zero. 
There is no functor 
\[
\mathcal{C}: \text{Top}_* \to CDGC_{\mathbf{k}}
\]
such that for any $(X,b) \in \text{Top}_*$, $\mathcal{C}(X,b)$ is $\Omega$-quasi-isomorphic to $C_*(X,b;\mathbf{k})$ as dg coassociative coalgebras. 
\end{theorem} 

\begin{proof}
Suppose there exists such a functor $\mathcal{C}$ and let $(X,b) \in \text{Top}_*$. Lemma \ref{lem:cobaranduniversalenvelopingalgebra} says that there is an isomorphism of dg associative algebras $$\Omega \mathcal{C}(X,b) \cong U \mathcal{L} \mathcal{C}(X,b),$$ so the homologies are isomorphic as algebras $$H_*(\Omega \mathcal{C}(X,b)) \cong H_*(U\mathcal{L}\mathcal{C}(X,b)).$$
 By Theorem \ref{thrm:Universalenvelopecommuteswithhomology}, we have an isomorphism of associative algebras
$$H_*(U\mathcal{L}\mathcal{C}(X,b)) \cong UH_*(\mathcal{L}\mathcal{C}(X,b)).$$
The PBW Theorem (Theorem \ref{thrm:PBW}) gives an isomorphism of graded vector spaces,
  $$UH_*(\mathcal{L}\mathcal{C}(X,b)) \cong S H_*(\mathcal{L}\mathcal{C}(X,b)).$$ 
In particular,  $$H_0(\Omega \mathcal{C}(X,b)) \cong S H_0(\mathcal{L}\mathcal{C}(X,b)).$$ 

By Corollary \ref{fundgroup}, our assumption yields an isomorphism of algebras $$\mathbf{k}[\pi_1(X,b)] \cong H_0(\Omega C_*(X,b;\mathbf{k})) \cong H_0(\Omega \mathcal{C}(X,b)) \cong S H_0(\mathcal{L}\mathcal{C}(X,b))$$ for all pointed path-connected spaces $(X,b)$. Note that $S H_0(\mathcal{L}\mathcal{C}(X,b))$ is either one or infinite dimensional as a vector space. Since $\pi_1(X,b)$ can be any arbitrary group and, in particular, a finite group has finite dimensional group algebra, we obtain a contradiction. 
\end{proof}

\begin{remark}
We would like to point out that a similar argument also implies that there is no functor
\[
\mathcal{C}: \text{Top}_* \to C_{\infty,\mathbf{k}},
\]
where $C_{\infty,\mathbf{k}}$ denotes the category of $C_\infty$-coalgebras over $\mathbf{k}$. A $C_\infty$-coalgebra is a cocommutative coalgebra whose coassociativity is relaxed up to homotopy, i.e. it is a coalgebra whose binary coproduct is strictly cocommutative, but only coassociative up to a sequence of coherent higher homotopies. The difference with an $E_\infty$-coalgebra is that in an $E_\infty$-coalgebra both the cocommutativity and coassociativity are relaxed up to homotopy, while in a $C_\infty$-coalgebra only the coassociativity is relaxed up to homotopy (see Sections 13.1.9 and 13.1.10 of \cite{LV12} for more details). Since there is a morphism of operads $A_\infty\rightarrow C_\infty$, we can also use the cobar construction for $A_\infty$-coalgebras for $C_\infty$-coalgebras. This cobar construction is defined in a similar way as for cocommutative algebra, i.e. it is defined as the tensor algebra generated by the coaugmentation ideal, with a differential coming from the coalgebra structure. For a $C_\infty$-coalgebra $C$, this cobar construction again factors as $\Omega C \cong U \mathcal{L} C$ ,where $\mathcal{L}$ now denotes the $C_\infty$ analog of $\mathcal{L}$ and $U$ is the universal enveloping algebra. By the PBW theorem this algebra will always be the symmetric algebra on some vector space $V$ and is therefore either one dimensional or infinite dimensional and can never  model finite fundamental groups, therefore showing that there is no functor $\mathcal{C}$ from connected spaces to  $C_\infty$-coalgebras such that $\mathcal{C}(X,b)$ is $\Omega$-quasi-isomorphic to $C_*(X,b;\mathbf{k})$ as  $A_\infty$-algebras.
\end{remark}

\bibliographystyle{plain}

\Addresses

\end{document}